\documentclass[12pt,a4paper,reqno]{article}
\pdfoutput=1
\usepackage{amsmath,amsfonts}
\usepackage{amsthm}
\usepackage{upref}
\usepackage{enumerate}
\usepackage{verbatim}
\usepackage{authblk}
\usepackage[margin=3.2cm]{geometry}
\usepackage[colorlinks=true]{hyperref}

\title{The role of domination and smoothing conditions in the theory of
  eventually positive semigroups}

\author[1]{Daniel Daners}%
\author[2]{Jochen Gl\"uck\thanks{Partially supported by a scholarship within the
    scope of the LGFG Baden-W\"urttemberg, Germany.}}%
\affil[1]{School of Mathematics and Statistics, University of Sydney,
  NSW 2006, Australia\authorcr%
  \nolinkurl{daniel.daners@sydney.edu.au}}%
\affil[2]{Institut f\"ur Angewandte Analysis, Universit\"at Ulm,
  D-89069 Ulm, Germany\authorcr%
  \nolinkurl{jochen.glueck@uni-ulm.de}}%
\date{\today}

\makeatletter
\hypersetup{pdfauthor={Daniel Daners, Jochen Glueck}}
\hypersetup{pdftitle={\@title}}
\makeatother

\newtheorem{theorem}{Theorem}[section]
\newtheorem{lemma}[theorem]{Lemma}

\newtheorem{corollary}[theorem]{Corollary}

\theoremstyle{definition}

\newtheorem*{definition*}{Definition}

\theoremstyle{remark}

\numberwithin{equation}{section}

\DeclareMathOperator{\im}{im}

\DeclareMathOperator{\spb}{s}
\DeclareMathOperator{\spec}{\sigma}
\DeclareMathOperator{\resSet}{\rho}
\DeclareMathOperator{\Res}{\mathcal{R}}

\DeclareMathOperator{\repart}{Re}

\DeclareMathOperator{\per}{per}
\DeclareMathOperator{\dist}{dist}

\newcommand{\bbN}{\mathbb{N}}

\newcommand{\bbR}{\mathbb{R}}

\newcommand{\calL}{\mathcal{L}}

\newcommand{\phdot}{\mathord{\,\cdot\,}}

\let\oldthebibliography\thebibliography
\renewcommand\thebibliography[1]{
  \oldthebibliography{#1}
  \setlength{\parskip}{0pt}
  \setlength{\itemsep}{0pt plus 0.3ex}
  \small
}

\begin{document}
\maketitle

\begin{abstract}
  We perform an in-depth study of some domination and smoothing
  properties of linear operators and of their role within the theory of 
  eventually positive operator semigroups. On the one hand we prove that, on many
  important function spaces, they imply compactness properties. On the
  other hand, we show that these conditions can be omitted in a number
  of Perron--Frobenius type spectral theorems.  We furthermore prove a
  Kre\u{\i}n--Rutman type theorem on the existence of positive
  eigenvectors and eigenfunctionals under certain eventual positivity
  conditions.
\end{abstract}

\renewcommand{\thefootnote}{}%

\section{Introduction}
\label{sec:introduction}
The solution of a linear autonomous evolution equation is often
described by means of a $C_0$-semigroup on a Banach space, usually some
kind of functions space. While, in many models, one expects the solution
semigroup to be \emph{positive}, that is, solutions with positive
initial conditions remain positive, there are also examples which
exhibit a more subtle type of positive behaviour. For example, it was
noted in \cite{Ferrero2008} and \cite{Gazzola2008} that the solution
semigroup of the bi-harmonic heat equation on $\bbR^d$, while not being
positive, behaves in some sense \emph{eventually} positive. This
observation complemented earlier results on the corresponding elliptic
problem; see for instance \cite{Grunau1997, Grunau1998, Grunau1999} and
the references therein and also the recent paper \cite{Sweers2016}. A
similar phenomenon occurs for the semigroup generated by the
Dirichlet-to-Neumann operator on a two-dimensional disk as shown in
\cite{Daners2014}.

These observations suggest that a general theory of eventually positive
$C_0$-semigroups would be useful. While, in finite dimensions, such a
theory has been developed during the last decade (see for instance
\cite{Noutsos2008,Olesky2009}, \cite[Theorem~2.9]{Ellison2009} and
\cite{Erickson2015}), a systematic study of this phenomenon in infinite
dimensions was initiated only recently in
\cite{Daners2016,Daners2016a}. Several spectral results for infinite
dimensional operators with eventually positive powers were recently
proved by the second author in \cite{GlueckEPO}, after eventually
positive matrix powers had been intensively studied for at least two
decades; see the introduction of \cite{GlueckEPO} for references and
additional details.

\paragraph*{A domination and a smoothing condition}
In the present note we are mainly concerned with two conditions
appearing in various characterisation theorems in
\cite{Daners2016a}. The conditions involve the \emph{principal ideal}
$E_u$ generated by some element $u$ of the positive cone $E_+$ of a
real or complex Banach lattice $E$. That principal ideal is defined by
\begin{displaymath}
  E_u := \{f \in E\colon\exists c \ge 0 \; |f| \le cu\}.
\end{displaymath}
It is a subspace of $E$ and when equipped with the \emph{gauge norm}
$\|\cdot\|_u$ given by
\begin{equation}
  \label{eq:gauge-norm}
  \|f\|_u:=\inf\{c \ge 0\colon |f| \le cu\}
\end{equation}
a Banach lattice in its own right. We will often assume that $u\in E_+$ is a
\emph{quasi-interior point} of the positive cone, that is, a point such
that $E_u$ is dense in $E$. We refer to
\cite{Meyer-Nieberg1991,Schaefer1974} for the general theory of Banach
lattices.

First condition: Given a linear operator $A\colon E \supseteq D(A) \to
E$ we refer to
\begin{equation}
  \label{eq:domination-condition}
  D(A)\subseteq E_u \tag{Dom}
\end{equation}
as the \emph{domination condition}. It plays an important role in the
characterisation of eventually positive behaviour of the resolvent of
$A$ in \cite[Theorem~4.4]{Daners2016a}. We call this a domination
condition since for every $v\in D(A)$ it implies the existence of $c>0$
such that $|v|\leq cu$.

Second condition: If $A$ generates a $C_0$-semigroup $(e^{tA})_{t \ge
  0}$ on $E$, then we refer to
\begin{equation}
  \label{eq:smoothing-condition}
  \exists t_0 \ge 0\colon \quad e^{t_0A}E \subseteq E_u \tag{Smo}
\end{equation}
as the \emph{smoothing condition}.  This condition is an important
assumption in \cite[Theorem~5.2]{Daners2016a} which characterises
eventual positivity of $(e^{tA})_{t \ge 0}$ by means of
Perron--Frobenius like properties.  We call
\eqref{eq:smoothing-condition} a smoothing condition since in general
the gauge norm is stronger than the norm induced by on $E$, and also
because $(E_u,\|\cdot\|_u)$ is isometrically Banach lattice isomorphic
to the space of real- or complex-valued continuous functions on some
compact Hausdorff space $K$. The latter follows from the corollary to
\cite[Proposition~II.7.2]{Schaefer1974} and from Kakutani's
representation theorem for AM-spaces
\cite[Theorem~2.1.3]{Meyer-Nieberg1991}.

If $E$ is the space of real- or complex-valued continuous functions on a compact
Hausdorff space $K$, endowed with the supremum norm, then we always have
$E_u = E$. Hence, conditions~\eqref{eq:domination-condition}
and~\eqref{eq:smoothing-condition} are automatically fullfilled on such spaces.
On many other Banach lattices, however, both conditions are quite strong.
In a typical application we can think of $E$ as an
$L^p$-space over a bounded domain $\Omega \subseteq \bbR^d$ with $1 < p
< \infty$ and of $A$ as a differential operator, defined on an
appropriate Sobolev space. The vector $u$ could, for instance, be the
constant function with value $1$ in which case $E_u$ coincides with
$L^\infty(\Omega)$. In this case the domination
condition~\eqref{eq:domination-condition} means that all functions in
the domain of $A$ are bounded; it is fulfilled if an appropriate Sobolev
embedding theorem holds. The smoothing
condition~\eqref{eq:smoothing-condition} means that the semigroup
operator $e^{t_0A}$ maps every function to a bounded function, that is,
it ``smooths'' unbounded initial data in some sense, see also the
comment above.

It should be noted that, for analytic semigroups,
condition~\eqref{eq:domination-condition}
implies~\eqref{eq:smoothing-condition}; see \cite[Remark~9.3.4]{GlueckDISS} 
or the proof of
\cite[Corollary~5.3]{Daners2016a} for details and for a slightly
stronger assertion. Given the fact that the
assumptions~\eqref{eq:domination-condition}
and~\eqref{eq:smoothing-condition} are fulfilled in many applications,
they were not studied in much detail in \cite{Daners2016a}; it was
merely demonstrated in \cite[Example~5.4]{Daners2016a} that these
conditions cannot be dropped in \cite[Theorems~4.4 and~5.2]{Daners2016a}
without one of the implications in those theorems failing.

\paragraph*{Aim of this note} The paper is devoted to an in-depth study
of the conditions~\eqref{eq:domination-condition}
and~\eqref{eq:smoothing-condition}. While, on spaces of continuous
functions over a compact space, both conditions are always fulfilled, we
will show in Section~\ref{section:domination-smoothing-and-compactness}
that both conditions are rather strong on other function spaces such as
the $L^p$-spaces. When $p \in [1,\infty)$, we see in
Corollary~\ref{cor:smoothing-condition-for-semigroups} that
condition~\eqref{eq:smoothing-condition} forces the semigroup
$(e^{tA})_{t \ge 0}$ to be eventually compact.

In Section~\ref{section:existence-of-positive-eigenvectors} we present a
short intermezzo on the existence of positive eigenvectors complementing
earlier results in \cite[Theorem~7.7.(i)]{Daners2016}. In
Sections~\ref{section:pf-theorem-for-resolvents}
and~\ref{section:pf-theorem-for-semigroups} we show that some of the
implications in the characterisation results in \cite[Theorems~4.4
and~5.2]{Daners2016a} remain true without the
conditions~\eqref{eq:domination-condition}
and~\eqref{eq:smoothing-condition}.

\paragraph*{Eventual positivity: terminology}
Several notions of eventual positivity were discussed in
\cite{Daners2016} and \cite{Daners2016a}, some of which we recall for
the convenience of the reader. For a concise formulation we introduce
some notation. Let $E$ be a real or complex Banach lattice. 
As usual we call $f \in E$ \emph{positive} if $f \ge 0$, and
we write $f>0$ if $f\geq 0$ but $f\neq 0$. If $u,f\in E_+$, then we write $f
\gg_u 0$ if there exists $\varepsilon > 0$ such that $f \ge\varepsilon
u$; in this case we call $f$ \emph{strongly positive with respect to
  $u$}. By $\calL(E)$ we denote the space of bounded linear operators on
$E$. An operator $T \in \calL(E)$ is called \emph{positive}, which we
denote by $T \ge 0$, if $TE_+ \subseteq E_+$. We call $T$ \emph{strongly
  positive} with respect to a vector $u \in E_+$ if $Tf \gg_u 0$ for
every $0 < f \in E_+$.

Now, let $E$ be a complex Banach lattice with real part $E_\bbR$ and let
$A \colon D(A) \to E$ be a linear operator. The operator $A$ is called 
\emph{real} if $D(A) = E_\bbR \cap D(A) + i E_\bbR \cap D(A)$ and if $A$ maps
$E_\bbR\cap D(A)$ to $E_\bbR$. 
The first notion of eventual positivity which we recall relates to the resolvent 
of $A$. We recall that
the resolvent $\lambda\mapsto\Res(\lambda,A) := (\lambda I- A)^{-1} \in 
\calL(E)$ is an analytic map on the \emph{resolvent set} $\resSet(A)$. 
We denote the spectrum of $A$ by $\spec(A)$.

\begin{definition*}
  Let $A\colon E \supseteq D(A) \to E$ be a linear operator on a complex
  Banach lattice $E$ and let $u \in E$ be a quasi-interior point of
  $E_+$. Let $\lambda_0 \in \bbR \cap \spec(A)$ be an isolated spectral
  value of $A$.
  \begin{enumerate}[(a)]
  \item The resolvent $\Res(\phdot,A)$ is called \emph{individually
      eventually strongly positive with respect to $u$ at $\lambda_0$}
    if, for every $0 < f \in E$, there exists a $\lambda_1 > \lambda_0$
    with the following properties: $(\lambda_0,\lambda_1] \subseteq
    \resSet(A)$ and $\Res(\lambda,A)f \gg_u 0$ for all $\lambda \in
    (\lambda_0,\lambda_1]$.
  \item The resolvent $\Res(\phdot,A)$ is called \emph{individually
      eventually strongly negative with respect to $u$ at $\lambda_0$}
    if, for every $0 < f \in E$, there exists a $\lambda_1 < \lambda_0$
    with the following properties: $[\lambda_1,\lambda_0) \subseteq
    \resSet(A)$ and $-\Res(\lambda,A)f \gg_u 0$ for all $\lambda \in
    [\lambda_1,\lambda_0)$
  \end{enumerate}
\end{definition*}
We speak of \emph{individual} eventual positivity as $\lambda_1$ can
depend on $f$. One can, of course, also define \emph{uniform} eventual
positivity; see \cite[Definitions~4.1 and~4.2]{Daners2016a} for details.
Note that if $\Res(\cdot,A)$ is eventually positive or negative at some
$\lambda_0\in\spec(A)$, then $A$ is \emph{real}, that is, $A$ leaves the
real part $E_\bbR$ of $E$ invariant.

The above definitions make sense even if $\lambda_0$ is not necessarily
an isolated point of $\spec(A)$, see \cite[Definitions~4.1
and~4.2]{Daners2016a}, but the above definition is sufficient for our
purposes. In fact we will usually assume that $\lambda_0$ is a pole of
the resolvent $\Res(\phdot,A)$ as an analytic map on $\resSet(A)$. Such
a pole is always an eigenvalue of $A$ as seen in \cite[Theorem 2 in
Section VIII.8]{Yosida1995}, and the pole is of order one if and only if
the geometric and algebraic multplicities of $\lambda_0$ as an
eigenvalue of $A$ coincide.

We next deal with $C_0$-semigroup on $E$ generated by an operator $A$ and 
denoted by $(e^{tA})_{t \ge 0}$.
\begin{definition*}
  Let $(e^{tA})_{t \ge 0}$ be a $C_0$-semigroup on a complex Banach
  lattice $E$ and let $u \in E$ be a quasi-interior point of $E_+$. The
  semigroup $(e^{tA})_{t \ge 0}$ is called \emph{individually eventually
    strongly positive with respect to $u$} if, for every $0 < f \in E$,
  there exists a time $t_0 \ge 0$ such that $e^{tA}f \gg_u 0$ for all $t
  \ge t_0$.
\end{definition*}
We talk about \emph{uniform eventual positivity} if $t_0$ can be chosen
independently of $f\in E_+$, see \cite[Definition~5.1]{Daners2016a} for
details. It is not difficult to see that $A$ is a
real operator if and only if the operator $e^{tA}$ is real for every $t
\in [0,\infty)$.

To a great extent the long-term behaviour of the semigroup is determined by
properties relating to the \emph{spectral bound} $\spb(A) := \sup
\{\repart \lambda\colon \lambda \in \spec(A)\} \in [-\infty,\infty]$ of
$A$. If $\spb(A) \in (-\infty,\infty)$, then of particular importance
is the \emph{peripheral spectrum} of $A$ given by $\spec_{\per}(A) :=
\{\lambda \in \spec(A)\colon \repart \lambda = \spb(A)\}$ and the
existence of a \emph{dominant spectral value}, that is,
$\lambda_0\in\spec(A)$ such that $\spec_{\per}(A) = \{\lambda_0\}$.

In Section~\ref{section:existence-of-positive-eigenvectors} we will also
encounter a slightly weaker notion of eventual positivity; see
Corollaries~\ref{cor:existence-of-positive-eigenvectors-resolvent}
and~\ref{cor:existence-of-positive-eigenvectors-semigroup} and the
preceeding discussions.

We complete this section by clarifying some notation we will use
throughout.  The \emph{dual space} of a real or complex Banach lattice 
$E$ is denoted by $E'$; it is also
a Banach lattice and its positive cone $E'_+$ is
called the \emph{dual cone} of $E_+$. A vector $\varphi \in E'$ is
positive if and only if $\langle \varphi, f \rangle \ge 0$ for all $0
\le f\in E$. Since $E'$ is a Banach lattice, all the notation introduced
above implies to the elements of this space, too; in particular, we
write $\varphi > 0$ if a functional $\varphi \in E'$ fulfils $\varphi
\ge 0$ but $\varphi \not= 0$. We call the functional $\varphi \in E'$
\emph{strictly positive} if $\langle \varphi, f\rangle > 0$ for all $0 <
f \in E$. Note that every quasi-interior point of $E'_+$ is a strictly
positive functional, but the converse is not in genera true. If $A\colon
E\supseteq D(A) \to E$ is a densely defined linear operator, then its dual
operator is denoted by $A': E' \supseteq D(A') \to E'$.

\paragraph*{Acknowledgement}
Some results presented here appear in the PhD Thesis of the second 
author \cite{GlueckDISS}, in particular 
Corollary~\ref{cor:existence-of-positive-eigenvectors-semigroup},
Theorem~\ref{thm:perron-frobenius-for-resolvent} and a weaker version
of Theorem~\ref{thm:composition-of-smoothing-operators}. We would like
to thank James Kennedy for useful discussions on some of the material.

\section{Domination, smoothing and compactness}
\label{section:domination-smoothing-and-compactness}
In this section we show that, on certain types of Banach lattices, the
conditions~\eqref{eq:domination-condition}
and~\eqref{eq:smoothing-condition} have rather strong consequences. Let
$E$ be a complex Banach lattice, let $u \in E_+$.
The fact that the gauge norm on $E_u$ is stronger than the
induced norm from $E$ has severe consequences on every operator $T \in
\calL(E)$ which maps $E$ to $E_u$ as we shall see in the main theorems of
this section.

To state the theorems we need to recall that a complex Banach lattice $E$
is said to have \emph{order continuous norm} if its real part $E_\bbR$
has order continuous norm. We refer to
\cite[Definition~2.4.1]{Meyer-Nieberg1991} for a precise definition. We
recall that every $L^p$-space with $1 \le p < \infty$ has order
continuous norm, as has the space $c_0$ of all real- or complex-valued
sequences which converge to $0$ endowed with the supremum norm. The
space of continuous functions on a compact Hausdorff space $K$ has never
order continuous norm unless $K$ is finite. We start with a lemma.
\begin{lemma}
  \label{lem:compactness-criterion}
  Let $E$ be a real or complex Banach lattice with order continuous norm
  and let $u \in E_+$.
  \begin{enumerate}[\upshape (i)]
  \item If $T\in\calL(E,E_u)$, then $T\in\calL(E)$ is weakly compact.
  \item If $T\in\calL(E,E_u)$ is weakly compact, then $T\in\calL(E)$ is
    compact.
  \end{enumerate}
\end{lemma}
\begin{proof}
  It suffices to prove the lemma in case that the scalar field is real.
  Since $E$ has order continuous norm every order interval in $E$ is
  weakly compact; see \cite[Theorem~2.4.2]{Meyer-Nieberg1991}. By
  definition of the gauge norm \eqref{eq:gauge-norm} every bounded set
  in $E_u$ is contained in an order interval in $E$. Hence the natural
  injection $j\colon E_u\rightarrow E$ given by $j(x)=x$ is weakly
  compact. If we are precise, then $T\in\calL(E)$ is the composition
  $j\circ T$.
  
  (i) As $T\colon E\to E_u$ is bounded and $j\colon E_u\to E$ is
  weakly compact we conclude that $T\colon E\to E$ is weakly compact.

  (ii) Because $E_u$ is a Dunford-Pettis space and $j\in\calL(E_u,E)$ is 
  weakly compact, $j$ is a Dunford--Pettis operator, that is,
  $x_n\rightharpoonup 0$ weakly in $E_u$ implies that $x_n\to 0$ in $E$;
  see Definition~3.7.6, Proposition~3.7.9 and Proposition~1.2.13 in 
  \cite{Meyer-Nieberg1991}.
  Let now $(x_n)$ be a bounded sequence in $E$. Then by the weak
  compactness of $T$ and the Eberlein--\v{S}mulian theorem
  \cite[Theorem~V.6.1]{dunford1958}, we can find a subsequence
  $(x_{n_k})$ such that $Tx_{n_k}\rightharpoonup y$ weakly in $E_u$ for
  some $y\in E_u$. Using that $j\in\calL(E_u,E)$ is a Dunford--Pettis
  operator we conclude that $(j\circ T)x_{n_k}\to Ty$ in $E$. This
  proves that $T\colon E\to E$ is a compact operator.
\end{proof}
\begin{theorem}
  \label{thm:composition-of-smoothing-operators}
  Let $E$ be a real or complex Banach lattice with order continuous norm
  and let $u \in E_+$. If $T_k\in\calL(E)$ and $T_k E \subseteq E_u$ for
  $k \in \{1,2\}$, then $T_2T_1\in\calL(E)$ is compact.
\end{theorem}
\begin{proof}
  First note that due to the closed graph theorem, $T_k\in\calL(E,E_u)$
  for $k=1,2$, where $E_u$ is as usual endowed with the gauge norm. By
  Lemma~\ref{lem:compactness-criterion}(i) $T_1\in\calL(E)$ is weakly
  compact. As $T_2\colon E\to E_u$ is continuous it is also weakly
  continuous and hence the composition $T_2T_1\colon E\to E_u$ is weakly
  compact. Now Lemma~\ref{lem:compactness-criterion}(ii) implies that
  $T_2T_1\in\calL(E)$ is compact.
\end{proof}
As a special case we can consider one operator $T_1=T_2=T$. If we assume 
that $E$ is reflexive, then we obtain an even stronger result.
Examples for reflexive Banach lattices are the $L^p$-spaces with
$1<p<\infty$ on an arbitrary measure space.
\begin{theorem}
  \label{thm:smoothing-operators}
  Let $E$ be a real or complex Banach lattice and let $u \in
  E_+$. Suppose that $T\in\calL(E)$ and that $TE \subseteq E_u$. Then
  the following assertions are true.
  \begin{enumerate}[\upshape (i)]
  \item If $E$ has order continuous norm, then $T^2\in\calL(E)$ is compact.
  \item If $E$ is reflexive, then $T\in\calL(E)$ is compact.
  \end{enumerate}
\end{theorem}
\begin{proof}
  (i) This is an obvious consequence of
  Theorem~\ref{thm:composition-of-smoothing-operators} taking
  $T_1=T_2=T$.

  (ii) First note that due to the closed graph theorem,
  $T\in\calL(E,E_u)$.  If $E$ is reflexive, then by the Banach-Alaoglu
  theorem every bounded set in $E$ is contained in a weakly compact
  set. As $T\in\calL(E,E_u)$ is continuous and thus weakly 
  continuous it follows that
  $T\in\calL(E,E_u)$ is weakly compact. Since it follows from
  \cite[Theorem~2.4.2(v)]{Meyer-Nieberg1991} that every reflexive Banach
  lattice has order continuous norm, we can now apply
  Lemma~\ref{lem:compactness-criterion}(ii) which shows that
  $T\in\calL(E)$ is compact.
\end{proof}
In \cite[Theorems~4.4 and~5.2]{Daners2016a} it was always assumed that
certain spectral values of $A$ be poles of the resolvent. In the
corollaries below we will show that the above results imply that such
assumptions are automatically satisfied if $E$ has order continuous norm
and if one of the conditions~\eqref{eq:domination-condition}
or~\eqref{eq:smoothing-condition} is fulfilled.  It is worthwhile
pointing out the the assumption of the first corollary is a bit more
general than condition~\eqref{eq:domination-condition}.
\begin{corollary}
  \label{cor:domination-condition-for-domain}
  Let $E$ be a complex Banach lattice, $u \in E_+$ and let $A\colon E
  \supseteq D(A) \to E$ be a linear operator with non-empty resolvent
  set. Suppose that $D(A^n) \subseteq E_u$ for some $n \in \bbN$. Then
  the following assertions are true.
  \begin{enumerate}[\upshape (i)]
  \item If $E$ has order continuous norm, then $\Res(\lambda,A)^{2n}$ is
    compact for every $\lambda \in \resSet(A)$.
  \item If $E$ is reflexive, then $\Res(\lambda,A)^{n}$ is compact for
    every $\lambda \in \resSet(A)$.
  \end{enumerate}
  In either case, all spectral values of $A$ are poles of the resolvent
  $\Res(\phdot,A)$ and have finite algebraic multiplicity.
\end{corollary}
\begin{proof}
  Let $\lambda \in \resSet(A)$. Then $\Res(\lambda,A)^nE = D(A^n)
  \subseteq E_u$. Now Theorem~\ref{thm:smoothing-operators}(i) and (ii)
  yield (i) and (ii) respectively. In either case
  \cite[Theorem~5.8-F]{Taylor1958} implies that all spectral values of
  $A$ are poles of $\Res(\phdot,A)$ and have finite algebraic
  multiplicity.
\end{proof}
Corollary~\ref{cor:domination-condition-for-domain} is useful to prove
that an operator in a concrete application has an eventually positive
resolvent. This can often be done by using
\cite[Theorem~4.4]{Daners2016a}. As it turns out, if $E$ has order
continuous norm and the domination
condition~\eqref{eq:domination-condition} is fulfilled, then, as a 
consequence of Corollary~\ref{cor:domination-condition-for-domain}, some of the
spectral theoretic assumptions in \cite[Theorem~4.4]{Daners2016a} are
automatically satisfied.

\begin{corollary}
  \label{cor:smoothing-condition-for-semigroups}
  Let $E$ be a complex Banach lattice with order continuous norm, $u \in
  E_+$ and let $(e^{tA})_{t \ge 0}$ be a $C_0$-semigroup on $E$. Suppose
  that $e^{t_0A}E \subseteq E_u$ for some $t_0 \ge 0$.

  Then the semigroup $(e^{tA})_{t \ge 0}$ is eventually compact. In
  particular, all spectral values of $A$ are poles of the resolvent
  $\Res(\phdot,A)$ and have finite algebraic multiplicity. Moreover, the
  peripheral spectrum of $A$ is finite.
\end{corollary}
\begin{proof}
  The semigroup is eventually compact since
  Theorem~\ref{thm:smoothing-operators} implies that the operator
  $e^{2t_0A}$ is compact. Hence, according to
  \cite[Corollary~V.3.2]{Engel2000}, all spectral values of $A$ are
  poles of $\Res(\phdot,A)$ and have finite algebraic multiplicity.
  It now follows from \cite[Theorem~II.4.18]{Engel2000} that the peripheral
  spectrum of $A$ is finite.
\end{proof}

As similar comment as given after
Corollary~\ref{cor:domination-condition-for-domain} also applies here.
In order to show that a given semigroup is eventually positive, one
can combine results from \cite[Section~5]{Daners2016a} with
Corollary~\ref{cor:smoothing-condition-for-semigroups}.

\section{The existence of positive eigenvectors}
\label{section:existence-of-positive-eigenvectors}
This section is devoted to a Kre\u{\i}n--Rutman type theorem about the
existence of positive eigenvectors. For eventually positive semigroup, a
related result was given in \cite[Theorem~7.7(i)]{Daners2016}. Similar
results for eventually and asymptotically positive operators can be
found in \cite[Section~6]{GlueckEPO}. The latter results also contain 
existence results about positive eigenvectors of the dual
operator. The following theorem and its corollaries are in the spirit of
this latter result. The proof of
Theorem~\ref{thm:existence-of-positive-eigenvectors} is inspired by the
proofs of \cite[Theorem~7.7(i)]{Daners2016} and
\cite[Theorem~6.1]{GlueckEPO}.
\begin{theorem}
  \label{thm:existence-of-positive-eigenvectors}
  Let $A\colon E \supseteq D(A) \to E$ be a linear operator on a complex
  Banach lattice $E$ and let $\lambda_0 \in \spec(A) \cap \bbR$ be a
  pole of the resolvent $\Res(\phdot,A)$. Suppose that we have, for
  every $f \in E_+$,
  \begin{equation}
    \label{eq:asymp-pos-resolvent}
    (\lambda - \lambda_0)\dist(\Res(\lambda,A)f,E_+) \to 0
  \end{equation}
  as $\lambda \downarrow \lambda_0$. Then the following assertions hold:
  \begin{enumerate}[\upshape (i)]
  \item The number $\lambda_0$ is an eigenvalue of $A$ and the
    corresponding eigenspace $\ker(\lambda_0 I- A)$ contains a positive,
    non-zero vector.
  \item If $A$ is densely defined, then $\lambda_0$ is an eigenvalue of
    the dual operator $A'$ and the corresponding eigenspace 
    $\ker(\lambda_0 I-A')$ contains a positive, non-zero vector.
  \end{enumerate}
\end{theorem}
We note in passing that in \cite{Daners2016a} the condition
\eqref{eq:asymp-pos-resolvent} is referred to as $\Res(\phdot,A)$ being
\emph{individually asymptotically positive at $\lambda_0$} if
$\lambda_0$ is a first-order pole of $\Res(\phdot,A)$.
\begin{proof}[Proof of Theorem~\ref{thm:existence-of-positive-eigenvectors}]
  (i) Let $m \in \bbN$, $m\geq 1$ denote the order of $\lambda_0$ as a pole
  of $\Res(\phdot,A)$ and let
  \begin{equation}
    \label{eq:laurent-expansion}
    \Res(\lambda,A) = \sum_{k=-m}^\infty (\lambda- \lambda_0)^k Q_k
  \end{equation}
  be the Laurent series expansion of $\Res(\phdot,A)$ about $\lambda_0$,
  where $Q_k \in \calL(E)$. Then $Q_{-m}\neq 0$ and
  $\im(Q_{-m})\subseteq\ker (\lambda_0 I-A)$; see \cite[Theorem~2 in
  Section~VIII.8]{Yosida1995}. In particular, $\lambda_0$ is an
  eigenvalue of $A$ and $(\lambda - \lambda_0)^m\Res(\lambda,A)\to
  Q_{-m}$ with respect to the operator norm as $\lambda \downarrow
  \lambda_0$. Hence Assumption~\eqref{eq:asymp-pos-resolvent} implies
  that $Q_{-m}$ is a positive operator. Since $Q_{-m}$ is non-zero its range
  contains a positive non-zero vector and this vector is an eigenvector
  of $A$ corresponding to the eigenvalue $\lambda_0$.

  (ii) Now assume that $A$ is densely defined so that it has a
  well-defined dual operator $A'$. Then $\Res(\lambda,A') =
  \Res(\lambda,A)'$ for all $\lambda \in \resSet(A) = \resSet(A')$, so
  it follows from~\eqref{eq:laurent-expansion} that the Laurent
  expansion of $\Res(\phdot,A')$ about $\lambda_0$ is given by
  \begin{displaymath}
    \Res(\lambda,A') = \sum_{k=-m}^\infty (\lambda- \lambda_0)^k Q_k'.
  \end{displaymath}
  In particular, as $Q_{-m}' \neq 0$, the point $\lambda_0\in\sigma(A')$
  is an $m$-th order pole of $\Res(\phdot,A')$. As $Q_{-m}$ is positive,
  so is $Q_{-m}'$ and hence, $\im(Q_{-m}')$ contains a positive non-zero
  vector. As $\im(Q_{-m}')\subseteq\ker(\lambda_0 I - A')$,
  this proves the assertion as in (ii).
\end{proof}
Let us formulate two corollaries where
Theorem~\ref{thm:existence-of-positive-eigenvectors} is applied to
eventually positive resolvents and to eventually positive semigroups.

First we recall the definition of an eventually positive resolvent from
\cite[Section~8]{Daners2016}. Let $A\colon E \supseteq D(A) \to E$ be a
linear operator on a complex Banach lattice $E$ and let $\lambda_0 \in
\spec(A) \cap \bbR$. We call the resolvent $\Res(\phdot,A)$ of $A$
\emph{individually eventually positive at $\lambda_0$} if for every $f
\in E_+$ there exists $\lambda_1 > \lambda_0$ such that
$(\lambda_0,\lambda_1] \subseteq \resSet(A)$ and $\Res(\lambda,A)f \ge
0$ for all $\lambda \in (\lambda_0,\lambda_1]$. The following corollary
is an immediate consequence of
Theorem~\ref{thm:existence-of-positive-eigenvectors}.
\begin{corollary}
  \label{cor:existence-of-positive-eigenvectors-resolvent}
  Let $A\colon E \supseteq D(A) \to E$ be a linear operator on a complex
  Banach lattice $E$ and let $\lambda_0 \in \spec(A) \cap \bbR$ be a
  pole of the resolvent $\Res(\phdot,A)$. Suppose that the resolvent of
  $A$ is individually eventually positive at $\lambda_0$.

  Then the assertions~{\upshape (i)} and~{\upshape (ii)} of
  Theorem~\ref{thm:existence-of-positive-eigenvectors} are fulfilled.
\end{corollary}

To formulate the second corollary, we recall the definition of an
eventually positive semigroup from \cite[Section~7]{Daners2016}. Let
$(e^{tA})_{t \ge 0}$ be a $C_0$-semigroup on a complex Banach lattice
$E$. We call this $C_0$-semigroup \emph{individually eventually
  positive} if, for every $f \in E_+$, there exists a time $t_0 \ge 0$
such that $e^{tA}f \ge 0$ for all $t \ge t_0$. We recall from
\cite[Theorem~7.6]{Daners2016} that the spectral bound $\spb(A)$ of the
generator of an individually eventually positive $C_0$-semigroup
$(e^{tA})_{t \ge 0}$ is always contained in the spectrum unless
$\spb(A)=-\infty$.

For individually eventually positive $C_0$-semigroups we obtain the
following corollary of
Theorem~\ref{thm:existence-of-positive-eigenvectors} which is a
generalisation of \cite[Theorem~7.7(a)]{Daners2016} in that it also
yields the existence of a positive eigenvector for the dual operator.

\begin{corollary}
  \label{cor:existence-of-positive-eigenvectors-semigroup}
  Let $(e^{tA})_{t \ge 0}$ be an individually eventually positive
  $C_0$-semigroup on a complex Banach lattice $E$. Suppose that
  $\spb(A)>-\infty$ is a pole of the resolvent $\Res(\phdot,A)$.

  Then the assertions~{\upshape (i)} and~{\upshape (ii)} of
  Theorem~\ref{thm:existence-of-positive-eigenvectors} are fulfilled for
  $\lambda_0 = \spb(A)$.
\end{corollary}
\begin{proof}
  Since the semigroup is individually eventually positive and since
  $\spb(A)>-\infty$ it follows from \cite[Corollary~7.3]{Daners2016}
  that the resolvent of $A$ fulfils
  property~\eqref{eq:asymp-pos-resolvent} in
  Theorem~\ref{thm:existence-of-positive-eigenvectors} for $\lambda_0 =
  \spb(A)$. Hence, the assertion follows from that theorem.
\end{proof}

\section{A Perron--Frobenius theorem for resolvents}
\label{section:pf-theorem-for-resolvents}
In this section we prove a Perron--Frobenius type theorem for eventually
positive resolvents. In contrast to the results of
Section~\ref{section:existence-of-positive-eigenvectors} we prove not
only the existence, but also the uniqueness of positive
eigenvectors. Let us start by recalling that a certain Perron--Frobenius
type property can be characterised by considering the spectral
projection of the eigenvalue under consideration. More precisely, let
$A\colon E \supseteq D(A) \to E$ be a real densely defined linear
operator $A\colon E \supseteq D(A) \to E$ on a complex Banach lattice
$E$, $\lambda_0\in\spec(A)\cap R$ a pole of $\Res(\phdot,A)$ and $u$ a
quasi-interior point of $E_+$. A typical conclusion of such a Perron-Frobenius 
type theorem is:
\begin{equation}
  \label{eq:perron-frobenius}
  \parbox{.87\textwidth}{
    The eigenvalue $\lambda_0$ of $A$ is geometrically simple and the
    corresponding eigen\-space $\ker(\lambda_0 I- A)$ contains a vector $v
    \gg_u 0$. Moreover, the eigenspace $\ker(\lambda_0 I- A')$ of the
    dual operator contains a strictly positive functional.}
\end{equation}
It was shown in \cite[Corollary~3.3]{Daners2016a} that a very concise
way of stating this conclusion is to say that the spectral projection
$P$ associated with $\lambda_0$ fulfills $P\gg_u0$.
Assertion~\eqref{eq:perron-frobenius} also implies
that $\lambda_0$ is algebraically simple and the only eigenvalue with a
positive eigenfunction.

It was further proved in \cite[Theorem~4.4]{Daners2016a} that, under
appropriate spectral assumptions combined with the domination
condition~\eqref{eq:domination-condition}, $P\gg_u0$ is equivalent to a
certain eventual positivity property of the resolvent
$\Res(\phdot,A)$. On the other hand, it was demonstrated in
\cite[Example~5.4]{Daners2016a} that such an equivalence is no longer true
if one drops the condition~\eqref{eq:domination-condition}. However, we
prove in the next theorem that some implications in
\cite[Theorem~4.4]{Daners2016a}, namely ``(ii) or (iii) $\Rightarrow$
(i)'', remain true without \eqref{eq:domination-condition}.
\begin{theorem}
  \label{thm:perron-frobenius-for-resolvent}
  Let $A\colon E \supseteq D(A) \to E$ be a densely defined and real
  linear operator on a complex Banach lattice $E$ and let $u \in E_+$ be
  a quasi-interior point. Assume that $\lambda_0 \in \spec(A) \cap \bbR$
  is a pole of the resolvent $\Res(\phdot,A)$ and denote the
  corresponding spectral projection by $P$.  If $\Res(\phdot,A)$ is
  individually eventually strongly positive or negative with respect to
  $u$ at $\lambda_0$, then $P \gg_u 0$.
\end{theorem}
The proof of the implications ``(ii) or (iii) $\Rightarrow$ (i)'' in
\cite[Theorem~4.4]{Daners2016a} cannot simply be adapted to work in our
more general setting here. The major obstacle is that
\cite[Lemma~4.8]{Daners2016a} relies on the domination
condition~\eqref{eq:domination-condition}. Here, we use a different
approach which has been inspired by the proof of
\cite[Proposition~B-III.3.5]{arendt1986}. We also need a simple
auxiliary result which was implicitly contained in the proof of
\cite[Lemma~7.4]{Daners2016}.
\begin{lemma}
  \label{lem:modulus-estimate}
  Let $E$ be a complex Banach lattice and let $(T_j)_{j \in J} \subseteq
  \calL(E)$ an individually eventually positive net of operators, in the
  sense that for all $f \in E_+$ there exists $j_0 \in J$ such that
  $T_jf \ge 0$ for all $j \ge j_0$.  Then, for every $f$ in the real
  part $E_\bbR$ of $E$, there exists $j_1 \in J$ such that $|T_jf| \le
  T_j|f|$ for all $j \ge j_1$.
\end{lemma}
\begin{proof}
  Choose $j_1 \in J$ such that $T_jf^+ \ge 0$ and $T_jf^- \ge 0$ for all
  $j \ge j_1$. For all those $j$ we then obtain $T_j(|f|+f) = 2T_jf^+
  \ge 0$ and $T_j(|f|-f) = 2T_jf^- \ge 0$. Hence, $T_j|f| \ge -T_jf$ and
  $T_j|f| \ge T_jf$ and thus $T_j|f| \ge |T_jf|$ for all $j \ge j_1$.
\end{proof}

\begin{proof}[Proof of Theorem~\ref{thm:perron-frobenius-for-resolvent}]
  We may assume throughout the proof that $\lambda_0 = 0$. Suppose that
  $\Res(\phdot,A)$ is individually eventually strongly positive with
  respect to $u$ at $\lambda_0 = 0$. We are going to show that
  \eqref{eq:perron-frobenius} is fulfilled.

  According to
  Corollary~\ref{cor:existence-of-positive-eigenvectors-resolvent} we
  can find vectors $0 < v \in \ker A$ and $0 < \varphi \in \ker A'$. We
  observe that \emph{every} element $0 < w \in \ker A$ fulfils $w \gg_u
  0$. Indeed, by assumption, for each such $w$ we can find a number
  $\lambda > 0$ for which $\lambda \in \resSet(A)$ and $w = \lambda
  \Res(\lambda,A)w \gg_u 0$. Therefore, $v \gg_u 0$.

  Next we show that the functional $\varphi$ is strictly positive. For
  every $0 < f \in E$ we can find a number $0<\lambda\in\rho(A)$ such that
  $\lambda \Res(\lambda,A)f \gg_u 0$; in particular, $\lambda
  \Res(\lambda,A)f$ is a quasi-interior point of $E_+$. Hence,
  \begin{displaymath}
    \langle \varphi, f \rangle = \langle \lambda
    \Res(\lambda,A')\varphi,f\rangle = \langle \varphi, \lambda
    \Res(\lambda,A)f \rangle > 0.
  \end{displaymath}
  Thus, $\varphi$ is indeed strictly positive.

  It remains to show that $\ker A$ is one-dimensional. To this end, we
  first prove that $E_\bbR \cap \ker A$ is a sublattice of the real part
  $E_\bbR$ of $E$. Fix $w \in E_\bbR \cap \ker A$. According to
  Lemma~\ref{lem:modulus-estimate} we can find a number $0<\lambda\in\rho(A)$
  such that $|w| = \lambda |\Res(\lambda,A)w| \leq \lambda
  \Res(\lambda,A)|w|$. By testing the positive vector $\lambda
  \Res(\lambda,A)|w| - |w|$ against the strictly positive functional
  $\varphi \in \ker A'$ we obtain
  \begin{displaymath}
    \langle \varphi, \lambda \Res(\lambda,A)|w| - |w| \rangle = \langle
    \lambda \Res(\lambda,A')\varphi, |w| \rangle - \langle
    \varphi,|w|\rangle = 0
  \end{displaymath}
  and thus, $\lambda \Res(\lambda,A)|w| = |w|$. This proves that $|w|
  \in E_\bbR \cap \ker A$, so $E_\bbR \cap \ker A$ is indeed a sublattice of
  $E_\bbR$.

  We have seen above that every non-zero positive vector in $w \in \ker
  A$ fulfils $w \gg_u 0$ and is thus a quasi-interior point of
  $E_\bbR$. Hence, according to \cite[Corollary 2 to Theorem
  II.6.3]{Schaefer1974}, $v$ is also a quasi-interior point of the
  positive cone of the Banach lattice $E_\bbR \cap \ker A$ (when endowed
  with the norm inherited from $E_\bbR$). We have thus shown that every
  positive non-zero element of the real Banach lattice $E_\bbR \cap \ker
  A$ is a quasi-interior point of its positive cone. This implies that
  $E_\bbR \cap \ker A$ is one-dimensional; see \cite[Lemma
  5.1]{Lotz1968} or \cite[Remark~5.9]{GlueckGR}. Since $A$ is real, we
  have $\ker A = E_\bbR \cap \ker A + iE_\bbR \cap \ker A$, so we
  conclude that $\ker A$ is one-dimensional over the complex field. This
  proves \eqref{eq:perron-frobenius}.

  Now assume instead that $\Res(\phdot,A)$ is individually eventually
  strongly negative with respect to $u$ at $\lambda_0 = 0$. Then the
  resolvent of $-A$, which is given by $\Res(\lambda,-A) =
  -\Res(-\lambda,A)$ for all $\lambda \in \resSet(-A) = -\resSet(A)$, is
  individually eventually strongly positive with respect to $u$ at
  $0$. Hence, by what we have just seen, the spectral projection of $-A$
  associated with $0$ is strongly positive with respect to $u$.  This
  spectral projection coincides with $P$, which proves the assertion
  by what we have shown above.
\end{proof}

\section{A Perron--Frobenius theorem for semigroups}
\label{section:pf-theorem-for-semigroups}

In this final section we pursue a similar goal as in
Section~\ref{section:pf-theorem-for-resolvents}, but this time for
eventually positive semigroups instead of resolvents. In
\cite[Theorem~5.2]{Daners2016a} it was shown that, under some
assumptions which include the smoothing
condition~\eqref{eq:smoothing-condition}, individual eventual strong
positivity with respect to $u$ of a semigroup $(e^{tA})_{t \ge 0}$ is
equivalent to a certain spectral condition that includes the
Perron--Frobenius properties discussed at the start of the previous
section. As demonstrated in \cite[Example~5.4]{Daners2016a} this results
fails in general if the smoothing
condition~\eqref{eq:smoothing-condition} is dropped. However, we are now
going to prove that at least a certain part of
\cite[Theorem~5.2]{Daners2016a} remains true without the
condition~\eqref{eq:smoothing-condition}.

\begin{theorem}
  \label{thm:perron-frobenius-for-semigroup}
  Let $(e^{tA})_{t \ge 0}$ be a real $C_0$-semigroup on a complex Banach
  lattice $E$ and let $u \in E_+$ be a quasi-interior point. Assume that
  $\spb(A)$ is not equal to $-\infty$ and a pole of the resolvent
  $\Res(\phdot,A)$.  If $(e^{tA})_{t\ge 0}$ is individually eventually
  strongly positive with respect to $u$, then the spectral projection
  $P$ corresponding to $\spb(A)$ fulfils $P \gg_u 0$.
\end{theorem}

For a similar reason as in
Section~\ref{section:pf-theorem-for-resolvents} we cannot simply modify
the relevant part of the proof of \cite[Theorem~5.2]{Daners2016a}.
Instead we adapt the argument in the proof of
Theorem~\ref{thm:perron-frobenius-for-resolvent} for semigroups.

\begin{proof}[Proof of Theorem~\ref{thm:perron-frobenius-for-semigroup}]
  We may assume throughout that $\spb(A) = 0$. According to
  Corollary~\ref{cor:existence-of-positive-eigenvectors-semigroup} there
  exists a vector $0 < v \in \ker A$ and a functional $0 < \varphi \in
  \ker A'$. To prove \eqref{eq:perron-frobenius} we now proceed
  similarly as in the proof of
  Theorem~\ref{thm:perron-frobenius-for-resolvent} with
  $\lambda_0=\spb(A)=0$.

  First note that every vector $0 < w \in \ker A$ fulfils $w \gg_u
  0$. Indeed, for each such vector we can find a time $t \ge 0$ for
  which we have $w = e^{tA}w \gg_u 0$. In particular we have $v \gg_u
  0$. Next we prove that the functional $\varphi$ is strictly
  positive. To this end, let $0 < x \in E$. We can find a time $t \ge 0$
  such that $e^{tA}x \gg_u 0$, so $e^{tA}x$ is a quasi-interior point of
  $E_+$. Hence, as $\varphi$ is non-zero we obtain
  \begin{displaymath}
    \langle \varphi, x \rangle = \langle (e^{tA})'\varphi,x\rangle
    = \langle \varphi, e^{tA}x \rangle > 0,
  \end{displaymath}
  which shows that $\varphi$ is indeed strictly positive.

  To conclude the proof, we still have to show that $\ker A$ is
  one-dimensional. As in the proof of
  Theorem~\ref{thm:perron-frobenius-for-resolvent}, let us first show
  that $E_\bbR \cap \ker A$ is a sublattice of $E_\bbR$. So, take $w \in
  E_\bbR \cap \ker A$ and choose a time $t_1 \ge 0$ such that $|w| = |e^{tA}w|
  \le e^{tA}|w|$ for all $t \ge t_1$; such a time $t_1$ exists according to
  Lemma~\ref{lem:modulus-estimate}. For $t \ge t_1$ we test the positive vector
  $e^{tA}|w| - |w|$ against the strictly positive functional $\varphi
  \in \ker A'$, thus obtaining
  \begin{displaymath}
    \langle \varphi, e^{tA}|w| - |w|\rangle = \langle
    (e^{tA})'\varphi,|w|\rangle - \langle \varphi,|w| \rangle = 0
  \end{displaymath}
  and hence $e^{tA}|w| = |w|$. For every $t \ge 0$ this implies $e^{tA}|w| =
  e^{tA}e^{t_1A}|w| = e^{(t+t_1)A}|w| = |w|$. Therefore, $|w| \in \ker A$, so 
  $E_\bbR \cap \ker A$
  is indeed a sublattice of $E_\bbR$. Now the same arguments as in the
  proof of Theorem~\ref{thm:perron-frobenius-for-resolvent} show that
  $\ker A$ is indeed one-dimensional.
\end{proof}
If the peripheral spectrum of $A$ is finite and consists of poles of the
resolvent and if the smoothing condition~\eqref{eq:smoothing-condition}
is fulfilled, then \cite[Theorem~5.2]{Daners2016a} asserts, among other
things, that individual eventual strong positivity of $(e^{tA})_{t \ge
  0}$ with respect to $u$ implies that the semigroup
$(e^{t(A-\spb(A))})_{t \ge 0}$ is bounded. It is an interesting question
whether this result remains true without the
condition~\eqref{eq:smoothing-condition}. This does not even seem to be
clear if the semigroup is strongly positive with respect to $u$, that
is, if $e^{tA} \gg_u 0$ for \emph{all} $t > 0$.

If, however, the semigroup under consideration is eventually norm
continuous, then the situation is much simpler. In this case we obtain
the following corollary which shows that the implication ``(i)
$\Rightarrow$ (ii)'' in \cite[Corollary~5.3]{Daners2016a} is true under
weaker assumptions than stated there.
\begin{corollary}
  \label{cor:perron-frobenius-for-semigroup}
  Let $(e^{tA})_{t \ge 0}$ be a real and eventually norm-continuous
  $C_0$-semigroup on a complex Banach lattice $E$ and let $u \in E_+$ be
  a quasi-interior point. Assume that $\spb(A) > -\infty$ and that the
  peripheral spectrum of $A$ is finite and consists of poles of the
  resolvent.

  If $(e^{tA})_{t\ge 0}$ is individually eventually strongly positive
  with respect to $u$, then the rescaled semigroup
  $(e^{t(A-\spb(A))})_{t \ge 0}$ is bounded, the spectral bound
  $\spb(A)$ is a dominant spectral value of $A$ and the corresponding
  spectral projection $P$ fulfils $P \gg_u 0$.
\end{corollary}
\begin{proof}
  We may assume that $\spb(A) = 0$. First recall from
  \cite[Theorem~7.6]{Daners2016} that $\spb(A)$ is a spectral value of
  $A$. It follows from Theorem~\ref{thm:perron-frobenius-for-semigroup}
  that $P \gg_u 0$. Hence, $\spb(A)$ is a first order pole of
  $\Res(\phdot,A)$ according to \cite[Corollary~3.3]{Daners2016a}, and
  this in turn implies that $\spec_{\per}(A)$ conists of first order
  poles of the resolvent; see \cite[Theorem~7.7(ii)]{Daners2016}.

  Now, let $\spec_{\per}(A) = \{i\beta_1,...,i\beta_n\}$ and denote by
  $Q \in \calL(E)$ the spectral projection of $A$ associated with
  $\spec_{\per}(A)$. Since $\spec_{\per}(A)$ consists of first order
  poles of the resolvent, we have $QE = \oplus_{k=1}^n \ker(i\beta_kI-
  A)$. Hence, the semigroup $(e^{tA})_{t \ge 0}$ is bounded on the range
  of $Q$. On the other hand, since the semigroup is eventually norm
  continuous and since $\spec_{\per}(A)$ is isolated from the rest of
  spectrum by assumption, it follows from
  \cite[Theorem~II.4.18]{Engel2000} that the spectral bound of
  $A|_{\ker Q}$ fulfils
  $\spb(A|_{\ker Q})<0$. Using again that our semigroup is eventually
  norm continuous, we conclude from \cite[Corollary~IV.3.11]{Engel2000}
  that $e^{tA}\to 0$ on $\ker Q$ with respect to the operator norm 
  as $t\to\infty$. Hence, $(e^{tA})_{t\ge 0}$ is indeed bounded as claimed.

  Finally, the boundedness and the individual eventual positivity of
  $(e^{tA})_{t \ge 0}$ imply immediately that this semigroup is
  \emph{individually asymptotically positive}, see
  \cite[Definition~8.1]{Daners2016a}. Hence, it follows from
  \cite[Theorem~8.3]{Daners2016a} that $\spb(A)$ is a dominant spectral
  value of $A$.
\end{proof}

\bibliography{}

\providecommand{\bysame}{\leavevmode\hbox to3em{\hrulefill}\thinspace}
\begin{thebibliography}{10}

\bibitem{arendt1986}
W.~Arendt, A.~Grabosch, G.~Greiner, U.~Groh, H.~P. Lotz, U.~Moustakas,
  R.~Nagel, F.~Neubrander, and U.~Schlotterbeck, \emph{One-parameter semigroups
  of positive operators}, Lecture Notes in Mathematics, vol. 1184,
  Springer-Verlag, Berlin, 1986.
  DOI:\,\href{http://dx.doi.org/10.1007/BFb0074922}{\nolinkurl{10.1007/BFb0074922}}

\bibitem{Daners2014}
D.~Daners, \emph{Non-positivity of the semigroup generated by the
  {D}irichlet-to-{N}eumann operator}, Positivity \textbf{18} (2014), 235--256.
  DOI:\,\href{http://dx.doi.org/10.1007/s11117-013-0243-7}{\nolinkurl{10.1007/s11117-013-0243-7}}

\bibitem{Daners2016a}
D.~Daners, J.~Gl{\"u}ck, and J.~B. Kennedy, \emph{Eventually and asymptotically
  positive semigroups on {B}anach lattices}, J. Differential Equations
  \textbf{261} (2016), 2607--2649.
  DOI:\,\href{http://dx.doi.org/10.1016/j.jde.2016.05.007}{\nolinkurl{10.1016/j.jde.2016.05.007}}

\bibitem{Daners2016}
D.~Daners, J.~Gl{\"u}ck, and J.~B. Kennedy, \emph{Eventually positive
  semigroups of linear operators}, J. Math. Anal. Appl. \textbf{433} (2016),
  1561--1593.
  DOI:\,\href{http://dx.doi.org/10.1016/j.jmaa.2015.08.050}{\nolinkurl{10.1016/j.jmaa.2015.08.050}}

\bibitem{dunford1958}
N.~Dunford and J.~T. Schwartz, \emph{Linear {O}perators. {I}. {G}eneral
  {T}heory}, Pure and Applied Mathematics, vol. VII, Interscience Publishers,
  Inc., New York, 1958.

\bibitem{Ellison2009}
E.~M. Ellison, L.~Hogben, and M.~J. Tsatsomeros, \emph{Sign patterns that
  require eventual positivity or require eventual nonnegativity}, Electron. J.
  Linear Algebra \textbf{19} (2009), 98--107.
  DOI:\,\href{http://dx.doi.org/10.13001/1081-3810.1350}{\nolinkurl{10.13001/1081-3810.1350}}

\bibitem{Engel2000}
K.-J. Engel and R.~Nagel, \emph{One-parameter semigroups for linear evolution
  equations}, Graduate Texts in Mathematics, vol. 194, Springer-Verlag, New
  York, 2000.
  DOI:\,\href{http://dx.doi.org/10.1007/b97696}{\nolinkurl{10.1007/b97696}}

\bibitem{Erickson2015}
C.~Erickson, \emph{Sign patterns that require eventual exponential
  nonnegativity}, Electron. J. Linear Algebra \textbf{30} (2015), 171--195.
  DOI:\,\href{http://dx.doi.org/10.13001/1081-3810.3027}{\nolinkurl{10.13001/1081-3810.3027}}

\bibitem{Ferrero2008}
A.~Ferrero, F.~Gazzola, and H.-C. Grunau, \emph{Decay and eventual local
  positivity for biharmonic parabolic equations}, Discrete Contin. Dyn. Syst.
  \textbf{21} (2008), 1129--1157.
  DOI:\,\href{http://dx.doi.org/10.3934/dcds.2008.21.1129}{\nolinkurl{10.3934/dcds.2008.21.1129}}

\bibitem{Gazzola2008}
F.~Gazzola and H.-C. Grunau, \emph{Eventual local positivity for a biharmonic
  heat equation in {$\mathbb R^n$}}, Discrete Contin. Dyn. Syst. Ser. S
  \textbf{1} (2008), 83--87.
  DOI:\,\href{http://dx.doi.org/10.3934/dcdss.2008.1.83}{\nolinkurl{10.3934/dcdss.2008.1.83}}

\bibitem{GlueckGR}
J.~Gl{\"u}ck, \emph{Growth rates and the peripheral spectrum of positive
  operators}, To appear in Houston Journal of Mathematics. Available at
  \href{https://arxiv.org/abs/1512.07483}{\nolinkurl{https://arxiv.org/abs/1512.07483}}

\bibitem{GlueckEPO}
J.~Gl{\"u}ck, \emph{Towards a {Perron}--{Frobenius} theory for eventually
  positive operators}, Preprint. Available at
  \href{https://arxiv.org/abs/1609.08498}{\nolinkurl{https://arxiv.org/abs/1609.08498}}

\bibitem{GlueckDISS}
J.~Gl{\"u}ck, \emph{Invariant sets and long time behaviour of operator
  semigroups}, Ph.D. thesis, Universit\"at Ulm, 2016.

\bibitem{Grunau1997}
H.-C. Grunau and G.~Sweers, \emph{Positivity for equations involving
  polyharmonic operators with {D}irichlet boundary conditions}, Math. Ann.
  \textbf{307} (1997), 589--626.
  DOI:\,\href{http://dx.doi.org/10.1007/s002080050052}{\nolinkurl{10.1007/s002080050052}}

\bibitem{Grunau1998}
H.-C. Grunau and G.~Sweers, \emph{The maximum principle and positive principal
  eigenfunctions for polyharmonic equations}, Reaction diffusion systems
  ({T}rieste, 1995), Lecture Notes in Pure and Appl. Math., vol. 194, Dekker,
  New York, 1998, pp.~163--182.

\bibitem{Grunau1999}
H.-C. Grunau and G.~Sweers, \emph{Sign change for the {G}reen function and for
  the first eigenfunction of equations of clamped-plate type}, Arch. Ration.
  Mech. Anal. \textbf{150} (1999), 179--190.
  DOI:\,\href{http://dx.doi.org/10.1007/s002050050185}{\nolinkurl{10.1007/s002050050185}}

\bibitem{Lotz1968}
H.~P. Lotz, \emph{\"{U}ber das {S}pektrum positiver {O}peratoren}, Math. Z.
  \textbf{108} (1968), 15--32.
  DOI:\,\href{http://dx.doi.org/10.1007/BF01110453}{\nolinkurl{10.1007/BF01110453}}

\bibitem{Meyer-Nieberg1991}
P.~Meyer-Nieberg, \emph{Banach lattices}, Universitext, Springer-Verlag,
  Berlin, 1991.
  DOI:\,\href{http://dx.doi.org/10.1007/978-3-642-76724-1}{\nolinkurl{10.1007/978-3-642-76724-1}}

\bibitem{Noutsos2008}
D.~Noutsos and M.~J. Tsatsomeros, \emph{Reachability and holdability of
  nonnegative states}, SIAM J. Matrix Anal. Appl. \textbf{30} (2008), 700--712.
  DOI:\,\href{http://dx.doi.org/10.1137/070693850}{\nolinkurl{10.1137/070693850}}

\bibitem{Olesky2009}
D.~D. Olesky, M.~J. Tsatsomeros, and P.~van~den Driessche,
  \emph{{$M_{\vee}$}-matrices: a generalization of {$M$}-matrices based on
  eventually nonnegative matrices}, Electron. J. Linear Algebra \textbf{18}
  (2009), 339--351.
  DOI:\,\href{http://dx.doi.org/10.13001/1081-3810.1317}{\nolinkurl{10.13001/1081-3810.1317}}

\bibitem{Schaefer1974}
H.~H. Schaefer, \emph{Banach lattices and positive operators}, Springer-Verlag,
  New York, 1974, Die Grundlehren der mathematischen Wissenschaften, Band 215.
  DOI:\,\href{http://dx.doi.org/10.1007/978-3-642-65970-6}{\nolinkurl{10.1007/978-3-642-65970-6}}

\bibitem{Sweers2016}
G.~Sweers, \emph{An elementary proof that the triharmonic {G}reen function of
  an eccentric ellipse changes sign}, Arch. Math. (Basel) \textbf{107} (2016),
  59--62.
  DOI:\,\href{http://dx.doi.org/10.1007/s00013-016-0909-z}{\nolinkurl{10.1007/s00013-016-0909-z}}

\bibitem{Taylor1958}
A.~E. Taylor, \emph{Introduction to functional analysis}, John Wiley \& Sons,
  Inc., New York; Chapman \& Hall, Ltd., London, 1958.

\bibitem{Yosida1995}
K.~Yosida, \emph{Functional analysis}, Classics in Mathematics,
  Springer-Verlag, Berlin, 1995.
  DOI:\,\href{http://dx.doi.org/10.1007/978-3-642-61859-8}{\nolinkurl{10.1007/978-3-642-61859-8}}

\end{thebibliography}

\end{document}